\documentclass[a4paper,12pt]{article}
\usepackage[utf8]{inputenc}
\usepackage{amssymb}
\usepackage{amsmath}
\usepackage{amsthm}
\usepackage{tikz}
\usepackage{url}
\theoremstyle{plain}
\newtheorem{defi}{Definition}[section]
\newtheorem{thm}{Theorem}[section]
\newtheorem{prop}{Proposition}[section]

\theoremstyle{remark}

\DeclareMathOperator{\Gal}{Gal}
\DeclareMathOperator{\Aut}{Aut}

\DeclareMathOperator{\Id}{Id}

\DeclareMathOperator{\md}{mod \;}
\DeclareMathOperator{\Der}{D}
\DeclareMathOperator{\Fr}{Frob}

\newcommand{\N}{\mathbb{N}}

\newcommand{\F}{\mathbb{F}}

\newcommand{\Proj}{\mathbb{P}^1}

\newcounter{mysub}

\begin{document}
\title{Big Actions with non abelian derived subgroup.}
\author{P. Chrétien and M. Matignon}

\maketitle

\begin{abstract}
For any $p>2$  we give an example of big action $(X,G)$ with non abelian derived
subgroup. It is obtained as a covering of a curve related to the Ree curve.
\end{abstract}
\section{Introduction}
\label{sec:0}
Let $k$ be an algebraically closed field of characteristic $p>0$, a \textsl{big action} is a pair $(X,G)$ where $X/k$ is a smooth, projective, integral curve of genus $g(X) \geq 2$ and $G$ is a finite $p$-group, $G \subseteq \Aut_k(X)$, such that $|G| > \frac{2p}{p-1}g(X)$. Big actions were studied by Lehr and Matignon \cite{LM} then by Matignon and Rocher \cite{MaRo} and Rocher \cite{Ro}. They study big actions $(X,G)$ with an abelian derived group $\Der(G)$. The main goal of this paper is to give the first example, to our knowledge, of big action $(X,G)$ with non abelian $\Der(G)$.

The approach in \cite{MaRo} to construct big actions $(X,G)$ with an abelian $\Der(G)$ is to consider ray class fields of function fields. Let $n \in \N-\lbrace 0 \rbrace$, $q:=p^n$, $m \in \N$, $K:=\F_q(x)$ and $S:= \lbrace (x-a), a\in \F_q \rbrace$ be the set of finite $\F_q$-rational places of $K$. One defines the \textsl{$S$-ray class field mod $m\infty$}, denoted by $K_S^{m\infty}$, as the largest abelian extension $L/K$ with conductor $\leq m\infty$ such that every place in $S$ splits completely in $L$. Denote by $G_S(m):=\Gal(K_S^{m\infty}/K)$ and $C_S(m)/\F_q$ the smooth, projective, integral curve with function field $K_S^{m\infty}/\F_q$. Then, the group of $\F_q$-automorphisms of $\Proj_{\F_q}$ given by $x \mapsto x+a$ with $a \in \F_q$ has a prolongation to a $p$-group $G(m) \subseteq \Aut_{\F_q}(C_S(m))$ with an exact sequence
\begin{align*}
0 \to G_S(m) \to G(m) \to \F_q \to 0.
\end{align*}
Moreover, if $m$ is large enough, then $|G(m)| > \frac{q}{-1+m/2}g(C_S(m))$. If moreover  $\frac{q}{-1+m/2} \geq \frac{2p}{p-1}$, then the pair $(C_S(m),G(m))$ is a big action and one can show that $\Der(G(m))=G_S(m)$.

The above construction leads to big actions $(X,G)$ with an abelian $\Der(G)$. In order to produce big actions with a non abelian derived subgroup, we are going to mimic this construction in a slightly different context. We construct a finite non abelian Galois extension $F/K$ with group $H:=\Gal(F/K)$ such that the group of $\F_q$-automorphisms of $\Proj_{\F_q}$ given by $x \mapsto x+a$ with $a \in \F_q$ has a prolongation to a $p$-group $G \subseteq \Aut_{\F_q}(F)$ with the exact sequence
\[
0 \to H \to G \to \F_q \to 0.
\]
Let $X/\F_q$ be the smooth, projective, integral curve with function field $F/\F_q$, our construction is such that $(X,G)$ is a big action and, as above, one can show that $\Der(G)=H$.

Let $s \in \N- \lbrace 0 \rbrace$ and $q:=3^{2s+1}$, the \textsl{Ree curve $X_R/\F_q$} has been extensively studied, see for example  \cite{Pe}, \cite{HaPe} and \cite{Lau1}. It is a Ray class field over $\F_q(x)$ and equations generalizing this situation for $p \geq 3$ are given in \cite{Au}. For $p=3$, the function field $F/\F_q$ is an extension of $F(X_R)/\F_q$.

\section{Background}
\label{sec:1}
\textbf{Notations :} Let $p$ be a prime number, $q:=p^n$ for some $n \in \N- \lbrace 0 \rbrace$ and et $k$ be an algebraically closed field of characteristic $p>0$.\\

\addtocounter{mysub}{1}
\arabic{mysub}. \textit{Galois Extensions of complete DVRs. } Let $(K,v_K)$ be a local field with uniformizing parameter $\pi_K$ such that $v_K(\pi_K)=1$. Let $L/K$ be a finite Galois extension with group $G$ and separable residual extension, denote by $v_L$ the prolongation of $v_K$ to $L$ such that $v_L(\pi_L)=1$ for some uniformizing parameter $\pi_L$ of $L$. Then $G$ is endowed with a \textsl{lower ramification filtration} $(G_i)_{i \geq -1}$ where $G_i$ is the \textsl{$i$-th lower ramification group} defined by $G_i := \lbrace \sigma \in G \; | \; v_{L}(\sigma(\pi_{L})-\pi_{L}) \geq i+1 \rbrace$. The integers $i$ such that $G_i \neq G_{i+1}$ are called \textsl{lower breaks}. The group $G$ is also endowed with a \textsl{higher ramification filtration} $(G^i)_{i \geq -1}$ which can be computed from the $G_i$'s by means of the \textsl{Herbrand's function} $\varphi_{L/K}$. The real numbers $t$ such that $ \forall \epsilon >0, \; G^{t+\epsilon} \neq G^t$  are called \textsl{higher breaks}. The least integer $m \geq 0$ such that $G^m=\lbrace 1 \rbrace$ is called the \textsl{conductor of $L/K$}. The following theorem is due to Garcia and Stichtenoth, see \cite{GS}.

\begin{thm}\label{GaSt}
Let $K$ be a perfect field of characreristic $p>0$. Let $F/K$ be an algebraic function field of one variable with full constant field $K$ and genus $g(F)$. Consider an elemantary abelian extension $E/F$ of degree $p^n$ such that $K$ is the constant field of $E$. Denote by $E_1, \dots, E_t$, where $t=(p^ n-1)/(p-1)$, the intermediate fields $F \subseteq E_i \subseteq E$ with $[E_i : F]=p$ and by $g(E)$ (resp. $g(E_i)$), the genus of $E/K$ (resp. $E_i/K$). Then
\begin{equation*}
g(E)=\sum_{i=1}^tg(E_i)- \frac{p}{p-1}(p^{n-1}-1)g(F).
\end{equation*}
\end{thm}
%\begin{lem}\label{lemdiscr}
%Let $K_1/K$ and $K_2/K$ be totally ramified linearly disjoint $p$-cyclic extensions with conductors $m_1<m_2$ and $M:=K_1K_2$. Then, the different ideal $\mathcal{D}_{M/K}$ of $M/K$ satisfies
%\[v_M(\mathcal{D}_{M/K})=(p-1)(m_2p+m_1).\]
%\end{lem}

\addtocounter{mysub}{1}

\arabic{mysub}. \textit{Automorphisms in positive characteristic. } See \cite{LM} for a complete account of big actions. Let $G$ be a group and $\Der(G)$ its derived subgroup.

\begin{defi}
A \textsl{big action} is a pair $(X,G)$ where $X/k$ is a smooth, projective, geometrically connected curve of genus $g(X) \geq 2$ and $G$ is a finite $p$-group, $G \subseteq \Aut_k(X)$, such that $|G| > \frac{2p}{p-1}g(X)$.
\end{defi}

\begin{prop}\label{BA}
Let $(X,G)$ be a big action and $H \subseteq G$ be a subgroup.
\begin{enumerate}
\item There exists a point of $X$, say $\infty$, such that $G$ is the wild inertia group $G_1$ of $G$ at $\infty$ and $\Der(G)$ is the second ramification subgroup $G_2$.
\item One has $g(X/H)=0$ if and only if $\Der(G) \subseteq H$.
\end{enumerate} 
\end{prop}

\begin{defi}Let $\pi : X \to \Proj_{\F_q}$ be a  smooth, projective, geometrically connected $p$-cyclic cover .  Then, $t_a \in \Aut_{\F_q}(\Proj_{\F_q})$  given by $x \mapsto x+a$ with $a \in \F_q$ has a prolongation $\tilde{t}_a \in \Aut_{\F_q}(X)$ if there is a commutative diagram 
\begin{center}
\begin{tikzpicture}[yscale=0.78]
\node (CM1) at (-1,1) {$X$};
\draw[->] (-0.5,1)--(0.5,1);
\node (CM2) at (1,1) {$X$};
\node (sigma1) at (-0.05,1.33) {$\tilde{t}_a$};
\draw[->] (-1,0.7)--(-1,-0.5);
\node (pi1) at (-1.2,0) {$\pi$};
\node (pi2) at (0.8,0) {$\pi$};
\draw[->] (1,0.7)--(1,-0.5);
\node (R1) at (-1,-1) {$ \Proj_{\F_q}$};
\draw[->] (-0.5,-1)--(0.5,-1);
\node (R2) at (1,-1) {$\Proj_{\F_q}$};
\node (sigma2) at (-0.05,-0.72) {$t_a$};
\end{tikzpicture}
\end{center}
\end{defi}
 
\addtocounter{mysub}{1}

%\arabic{mysub}. \textit{Maximal curves. } Let $X/\F_q$ be a smooth, projective, integral curve of genus $g(X)$.
%
%\begin{defi}
%\begin{enumerate}
%\item Let $N(X)$ be the number of $\F_q$-rational points of $X/\F_q$.
%\item For any $g \in \N$, let $N_q(g)$ denote the maximum number of $\F_q$-rational points that a smooth, projective, integral curve $Y/\F_q$ of genus $g$ can have.
%\item The curve $X/\F_q$ is called \textsl{optimal} if $N(X)=N_q(g(X))$.
%\item The curve $X/\F_q$ is called \textsl{maximal} if $N(X)=q+1+2g(X)\sqrt{q}$.
%\item For any prime power $q$, let
%\[ A(q):= \limsup_{g \to \infty} \frac{N_q(g)}{g}. \]
%\end{enumerate}
%\end{defi}
%
%
%
%\begin{prop}\label{boundAq}
%\begin{enumerate}
%\item (Vladut-Drinfeld Bound). For any prime power $q$ one has 
%\[ A(q) \leq \sqrt{q}-1.\]
%\item (see \cite{NX0} Corollary 5.3.2) If $q=p^e$ with an odd prime $p$ and an odd integer $e \geq 3$, then
%\[ A(q) \geq \frac{2q^{1/m}+2}{ \lceil 2(2q^{1/m}+3)^{1/2} \rceil +1}, \]
%where $m$ is the least prime dividing $e$.
%\end{enumerate}
%\end{prop} 
% 
%\addtocounter{mysub}{1}

\arabic{mysub}. \textit{Ree Curves. } There are three types of irreducible curves arising as the Deligne-Lusztig variety associated to  a connected, reductive, algebraic group, these are the Hermitian curves, the Suzuki curves and the 
Ree curves (see \cite{Lau1}). In this paper we will focus on the Ree curves which have been described in \cite{Pe} and \cite{HaPe}. Let $s \in \N - \lbrace 0 \rbrace$, $q_0:=3^s$ and $q:=3q_0^2$.

\begin{defi} The \textsl{Ree curves} are the Deligne-Lusztig varieties $X_R/\F_q$ arising from the Ree groups $^2G_2(q)$. 
\end{defi}

\begin{prop} The Ree curve $X_R/\F_q$ is an irreducible curve of genus $g(X_R)=\frac{3}{2}q_0(q-1)(q+q_0+1)$.
\begin{enumerate}
\item[a)] The function field $\F_q(X_R)$ is $\F_q$-isomorphic to $\F_q(x,y_1,y_2)$ defined by
\[
 \left\{
    \begin{array}{ll}
        y_1^q-y_1 &= x^{q_0}(x^q-x) \\
        y_2^q-y_2 &= x^{2q_0}(x^q-x).
    \end{array}
    \right.
\]
\item[b)] The curve $X_R/\F_q$ is optimal. The curve $X_R/\F_{q^n}$ is maximal if and only if $n \equiv 6 \md 12$.
\item[c)] The ramification filtration of $G:=\Gal(\F_q(x,y_1,y_2)/\F_q(x))$ at $\infty$ is
\begin{align*}
G=G_0=\dots=G_{3q_0+1} \supsetneq G_{3q_0+2}=\dots=G_{q+3q_0+1} \supsetneq \lbrace 1 \rbrace,
\end{align*}
and $G_{q+3q_0+1}=\Gal(\F_q(x,y_1,y_2)/\F_q(x,y_1))$.
\item[d)] One has $|\Aut_{\F_q^{\rm alg}}(X_R)|=|\Aut_{\F_q}(X_R)|=q^3(q-1)(q^3+1)$  and the $3$-Sylow subgroups $S_3(X_R)$ of $\Aut_{\F_q}(X_R)$ are such that $(X_R,S_3(X_R))$ are big actions.
\end{enumerate}
\end{prop}

The function field $\F_q(X_R)/\F_q$ and its subextension $\F_q(x,y_1)/\F_q$ have a description as Ray class fields, see below.\\

\addtocounter{mysub}{1}

\arabic{mysub}. \textit{Ray Class Fields. } See \cite{Au} for a detailed account.  Let $K:=\F_q(x)$ and fix an algebraic closure $K^{\rm alg}$ in which all extensions of $K$ are assumed to lie. In this paper, we consider only Galois extensions of function fields of one variable with full constant field $\F_q$ that are totally ramified over a $\F_q$-rational point, say $\infty$, and unramified outside $\infty$. In this setting, the definition of the conductor given above coincide with that given in \cite{Au} Part I.3.

\begin{defi}
Let $S:= \lbrace (x-a), a\in \F_q \rbrace$ be the set of finite $\F_q$-rational places of $K$ and $ m \in \N$. One defines the \textsl{$S$-ray class field mod $m\infty$}, denoted by $K_S^{m\infty}$, as the largest abelian extension $L/K$ with conductor $\leq m\infty$ such that every place in $S$ splits completely in $L$.
\end{defi}

\begin{prop}[\cite{Au} III. Prop. 8.9 b) and Lemma 8.7 c)] \label{propAuer}
Assume that $r:=\sqrt{pq} \in \N$ and let $y_1, \dots, y_{p-1} \in K^{\rm alg}$ satisfy $y_i^q-y_i=x^{ir/p}(x^q-x)$. Then $K_S^{i\infty}=K$ for $1 \leq i \leq p$ and
\[ K_S^{(r+i+1)\infty}=K(y_1,\dots,y_i) \; \; {\rm for} \; \; i \in \lbrace 1, \dots, p-1 \rbrace.\]
\end{prop}

\section{Results}
\label{sec:1}
\textbf{Notations : } Let $p>2$ be a prime number, $s \in \N - \lbrace 0 ,1  \rbrace$, $q_0:=p^s$ and $q:=pq_0^2$. Let $(\gamma_i)_{i=1}^{2s+1}$ be a $\F_p$-basis of $\F_q$,
\begin{align*}
 \Fr_p : K^{\rm alg}& \longrightarrow K^{\rm alg}\\
       x &  \longmapsto x^p,
\end{align*}
 and $\Fr_q=\Fr_p^{2s+1}$. Let $K:=\F_q(x)$ and $F/\F_q$ be the function field of one variable with full constant field $\F_q$ defined by 
\[
 \left\{
    \begin{array}{ll}
        y_1^q-y_1 &= x^{q_0}(x^q-x) =:f_1(x) \\
        y_2^q-y_2 &= x^{2q_0}(x^q-x) =: f_2(x)\\
        v_1^q-v_1 &= y_1^qx-x^qy_ 1\\
        v_2^q-v_2 &= y_2^qx-x^qy_ 2\\
        w^q-w     &=f_2(x)y_1-f_1(x)y_2 = y_2^qy_1-y_1^qy_ 2.
    \end{array}
    \right.
\]
\textbf{Remark : } The function field $F/\F_q$ is also defined by the equations 
\[
 \left\{
    \begin{array}{ll}
        y_1^q-y_1 &= x^{q_0}(x^q-x) =:f_1(x) \\
        y_2^q-y_2 &= x^{2q_0}(x^q-x) =: f_2(x)\\
        v'^q_1-v'_1 &= x^{q_0}(x^{2q}-x^2)=:g_1(x)\\
        v'^q_2-v'_2 &= x^{2q_0}(x^{2q}-x^2) =:g_2(x)\\
        w'^q-w'     &=2y_1f_2(x)+f_1(x)f_2(x),
    \end{array}
    \right.
\]
allowing us to view $F$ as the compositum of extensions of $\F_q(x,y_1)$.

\begin{thm}\label{thm1}
Let $X/\F_q$ be the smooth, projective, integral curve with function field $F/\F_q$. Let ${H \subseteq \Aut_{\F_q}(K)}$ be the subgroup of translations  $x \mapsto x+a$, ${a \in \F_q}$,  then any $h \in H$ has $q^5$ prolongations to $F$, the extension $F/K^H$ is Galois, the group $G:=\Gal(F/K^H)$ has order $q^6$, the pair $(X,G)$ is a big action and $\Der(G)$ is a non-abelian group.
\end{thm}

\begin{proof} Let $a \in \F_q$ and $t_a \in \Aut_{\F_q}(K)$ given by $x \mapsto x+a$. Let $\sigma : F \hookrightarrow K^{\rm alg}$ be a morphism such that $\sigma|_K=t_a$, an easy computation shows that 
\begin{align*}
(\sigma(y_1)^q-\sigma(y_1))-(y_1^q-y_1),& \; \; (\sigma(y_2)^q-\sigma(y_2))-(y_2^q-y_2),\\
(\sigma(v'_1)^q-\sigma(v'_1))-(v'^q_1-v'_1),& \; \;(\sigma(v'_2)^q-\sigma(v'_2))-(v'^q_2-v'_2),\\
 (\sigma(w')^q-\sigma(w'))&-(w'^q-w'), 
\end{align*}
are in $\Fr_q(F)$, thus the elements of $H$ have $q^5$ prolongations to $F$ and $F/K^H$ is a Galois extension of degree $q^6$. 

Using Theorem \ref{GaSt}, one computes the genus of $K(y_1,y_{2,i})$ defined by the equations
\[
 \left\{
    \begin{array}{ll}
        y_1^q-y_1 &=f_1(x) \\
        y_{2,i}^p-y_{2,i} &=\gamma_if_2(x).
          \end{array}
    \right.
\]
%It is enough to compute the genus of the function field $K(y_{1,j},y_{2,i})$ defined by the equations
%\[
% \left\{
%    \begin{array}{ll}
%        y_{1,j}^p-y_{1,j} &=\gamma_jf_1(x) \\
%        y_{2,i}^p-y_{2,i} &=\gamma_if_2(x).
%          \end{array}
%    \right.
%\]
%Applying lemma \ref{lemdiscr} with  $m_1=1+q/q_0<m_2=2+q/q_0$ and the Riemann-Hurwitz formula one obtains that  
%\[g(K(y_{1,j},y_{2,i}))=\frac{p-1}{2}\big[p+\frac{q}{q_0}(1+p)\big], \]
%then, \cite{GS} Theorem 2.1 yields
One obtains 
\[  g(K(y_1,y_{2,i}))=\frac{q}{2q_0}\big[qp+q_0p-q_0-1\big]. \]
Let $K(y_1,v'_{1,i})$ and $K(y_1,v'_{2,i})$ be the function fields defined by the equations

\begin{minipage}[t]{0.45\linewidth}
\[
 \left\{
    \begin{array}{ll}
        y_1^q-y_1 &=f_1(x) \\
        v'^p_{1,i}-v'_{1,i} &=\gamma_ig_1(x),
          \end{array}
    \right.
\]
\end{minipage}\begin{minipage}[t]{0.45\linewidth}
\[
 \left\{
    \begin{array}{ll}
        y_1^q-y_1 &=f_1(x) \\
        v'^p_{2,i}-v'_{2,i} &=\gamma_ig_2(x).
          \end{array}
    \right.
\]
\end{minipage}\\

\noindent One computes their genera as above and one obtains
\begin{align*}
g(K(y_1,v'_{1,i}))&=\frac{q}{2q_0}\big[2qp-q-1\big],\\
g(K(y_1,v'_{2,i}))&=\frac{q}{2q_0}\big[2qp+q_0p-q_0-q-1\big]. 
\end{align*}

Let $K(y_1,w'_i)$ be the function field defined by the equations
\[
 \left\{
    \begin{array}{ll}
        y_1^q-y_1 &=f_1(x) \\
        w'^p_i-w'_i &=\gamma_i\big[2y_1f_2(x)+f_1(x)f_2(x)\big]=\gamma_iF(z),
          \end{array}
    \right.
\]
in order to compute its genus, one needs an expression of $y_1$ and $x$ in terms of a uniformizing parameter of $K(y_1)$ at infinity, this is the crucial point of the proof. One defines $z $ by
\begin{equation*}
a_1:=\frac{q^2-qq_0-q}{q_0}, \; a_2:=\frac{q^2-q_0-q}{q_0} \; {\rm and \;} x=z^{-q}+z^{a_1}-z^{a_2}.
\end{equation*}
Then, one shows that this change of variable completely splits the place $x=\infty$ in $K(y_1)$. One puts
\begin{align*}
b_1:=&a_1-qq_0, \; \; \; b_2:=a_2-qq_0,\\
y_T:=&\frac{1}{z^{q+q_0}}+z^{b_1}-z^{b_2}+z^{a_1q_0-q}-z^{a_2q_0-q}\\
&+z^{a_1(1+q_0)}+z^{a_2(1+q_0)}-z^{a_1+a_2q_0}-z^{a_1q_0+a_2}+T.
 \end{align*}
By expanding $y_T^q-y_T-f_1(x)$ one gets for some $G(z) \in \F_q[|z|]$ 
\begin{equation*}\label{eq:yT}
y_T^q-y_T-f_1(x)=z^{qb_1}(1+zG(z))+T^q-T.
\end{equation*}
According to Hensel's lemma, the equation $T^q-T+z^{qb_1}(1+zG(z))=0$ in $\F_q[|z|][T]$ has a solution $T_0 \in \F_q[|z|]$ such that $v_z(T_0)>0$, thus $v_z(T_0)=qb_1$. So one has constructed a solution $y_{T_0} \in \F_q[|z|]$ to the equation  $Y^q-Y=f_1(x)$. Whence, one has the following diagram

\begin{center}
\begin{tikzpicture}[scale=1]

\draw[-] (-0.5,-0.6)--(-1.5,0.7);
\draw[-] (0.4 ,-0.6)--(1.5,0.7);

\node (Fqx) at (0,-1) {$\F_q((\frac{1}{x}))$};
\node (Fqy) at (1.5,1)  {$\F_q((\frac{1}{x}))[y_{T_0}]$};
\node (Fqz) at (-1.5,1) {$\F_q((z))$};
\node (inc) at (-0.2,1) {$\supseteq$};

%\node (deg1) at (-2.5 , 0) {$[\F_q((z)): \F_q((\frac{1}{x}))]=q$};
%\node (deg2) at (1.9,0 ) {$[\F_q((\frac{1}{x}))[y_{T_0}]: \F_q((\frac{1}{x}))]=q$};
\end{tikzpicture}
\end{center}
Since $[\F_q((z)): \F_q((\frac{1}{x}))]=q$ and $[\F_q((\frac{1}{x}))[y_{T_0}]: \F_q((\frac{1}{x}))]=q$, one has ${\F_q((\frac{1}{x}))[y_{T_0}]=\F_q((z))}$, i.e. $z$ is a uniformizing parameter of $K(y_1)$ at infinity. Note that letting $y_T:=\frac{1}{z^{q+q_0}}+z^{b_1}+T$ and using the same process, one still obtains that $z$ is a uniformizing parameter of $K(y_1)$ at infinity, but in this case one has $v(T_0)=b_2$ and one needs a more accurate expansion of $y_1$ in order to compute the genus of $K(y_1,w'_i)$, see below. 

Then, one expands $\gamma_iF(z) \in \F_q((z))$ in terms of $z$ and $T_0$ and one reads its principal part $P_i(z)$. Note that $v_z(T_0)=qb_1$ implies that the terms in $\gamma_iF(z)$ where $T_0$ appears do not disturb $P_i(z)$. 
%For $p=3$, one has
%\begin{align*}
% P(z) \equiv &-\frac{1}{z^{1+3q_0+q-a_2+3q_0q}}- \frac{1}{z^{1+3q_0+2q+3q_0q}} \\
% &+ \frac{1}{z^{1+2q_0}}+\frac{1}{z^{1+3q_0+q}}\md (\Fr_p-\Id)(\F_q((z))),
% \end{align*}
One has 
\begin{equation*}
 P_i(z) = \gamma_i\Big[\frac{1}{z^{3q_0q+2q^2}}+\frac{1}{z^{q^2+q+3q_0q}}- \frac{1}{z^{q_0+q+qq_0-a_2q_0+q^2}}-\frac{1}{z^{q_0+q+2qq_0+q^2}}\Big],
\end{equation*}
and $\md (\Fr_p-\Id)(\F_q((z)))$
\begin{equation*}
 P_i(z) \equiv \frac{\gamma_i^{q/q_0}}{z^{3+2pq_0}}+\frac{\gamma_i}{z^{1+3q_0+q}}- \frac{\gamma_i^{q/q_0}}{z^{1+pq_0+q-a_2+pq_0q}}-\frac{\gamma_i^{q/q_0}}{z^{1+pq_0+2q+pq_0q}}.
\end{equation*}
Thus, the conductor of the extension $K(y_1,w'_i)/K(y_1)$ is $2+pq_0+2q+pq_0q$ and applying the Riemann-Hurwitz formula, one obtains
\[ g(K(y_1,w'_i))=\frac{q}{2q_0} \big[2pq+2pq_0-q_0-q-1 \big]. \]
Applying \cite{GS} Theorem 2.1, one obtains that $g(F)$ equals
\begin{align} \label{genus}
 \frac{q-1}{p-1}&\big[ q^3g(K(y_1,w'_i))+q^2g(K(y_1,v'_{2,i}))+qg(K(y_1,v'_{1,i})) \nonumber \\
&+g(K(y_1,y_{2,i})) \big]-\frac{q-1}{p-1}\frac{q}{2q_0}(q-1).
 \end{align}
Then, an easy computation shows that $(X,G)$ is a big action. Actually, the leading term in equation \eqref{genus} is $\frac{q-1}{p-1}q^3g(K(y_1,w'_i))$ which, surprisingly, is not too large compared to $|G|$, that is why $(X,G)$ is a big action (note that $\lim_ {p\to \infty}\frac{|G|}{g(F)} = q_0$, checking the inequality  $|G|>\frac{2p}{p-1}g(F)$ being left to the reader). 

One shows that $\Der(G)=\Gal(F/K)$. Let $L:=F^{\Der(G)}$, according to Proposition \ref{BA} \textit{2}, one has $\Der(G) \subseteq \Gal(F/K)$, whence $K \subseteq L$. According to Proposition \ref{BA} \textit{2}, the function field $L$ has genus $0$, so the Riemann-Hurwitz formula implies that the conductor of $L/K$ is $\leq 2\infty$. Let $S$ be the set of finite $\F_q$-rational places of $K$, i.e. $S:= \lbrace (x-a), a \in \F_q \rbrace$. Then $L/K$ is an abelian extension with conductor $\leq 2\infty$ such that every place in $S$ splits completely in $L$, then $L \subseteq   K_S^{2\infty}$. According to Proposition \ref{propAuer}, $K_S^{2\infty}=K$, i.e. $L=K$ and $\Der(G)=\Gal(F/K)$.

One shows that $\Der(G)=\Gal(F/K)$ is non abelian. The $K$-automorphisms of $F/K$ defined by\\

   \begin{minipage}[c]{.46\linewidth}
     \[
 \left\{
    \begin{array}{ll}
       \sigma_i(y_1) &= y_1+\gamma_i\\
       \sigma_i(y_2) &= y_2 \\
       \sigma_i(v'_1) &= v'_1+\gamma_i\\
       \sigma_i(v'_2) &= v'_2 \\
       \sigma_i(w) &= w+\gamma_iy_2 \\
    \end{array}
    \right.
\]
   \end{minipage} and 
   \begin{minipage}[c]{.46\linewidth}
    \[
 \left\{
    \begin{array}{ll}
       \tau_i(y_1) &= y_1\\
       \tau_i(y_2) &= y_2+\gamma_i \\
       \tau_i(v'_1) &= v'_1+\gamma_i\\
       \tau_i(v'_2) &= v'_2 \\
       \tau_i(w) &= w-\gamma_iy_1 \\
    \end{array}
    \right.
\]
   \end{minipage}\\
are such that $\tau_i$ and $\sigma_{j}$ do not commute since $p > 2$.
\end{proof}

\bibliographystyle{alpha} 
\bibliography{rcfutf}

\begin{thebibliography}{Aue99}

\bibitem[Aue99]{Au}
R.~Auer.
\newblock {\em Ray Class Fields of Global Function Fields with Many Rational
  Places}.
\newblock PhD thesis, Lindau/ Bodensee, 1999.

\bibitem[GS91]{GS}
A.~Garcia and H.~Stichtenoth.
\newblock {\em Elementary abelian p-extensions of algebraic function fields}.
\newblock {\em {\rm Manuscripta Mathematica}}, ({\bf 72}), 1991.

\bibitem[HP93]{HaPe}
J.P Hansen and J.P. Pedersen.
\newblock {\em Automorphism groups of Ree type, Deligne-Lusztig curves and
  function fields}.
\newblock {\em {\rm J. reine angew. Math.}}, ({\bf 440}), 1993.

\bibitem[Lau99]{Lau1}
K.~Lauter.
\newblock {\em Deligne-Lusztig curves as ray class fields}.
\newblock {\em {\rm Manuscripta Mathematica}}, ({\bf 98}), 1999.

\bibitem[LM05]{LM}
C.~Lehr and M.~Matignon.
\newblock {\em Automorphism groups for $p$-cyclic covers of the affine line}.
\newblock {\em {\rm Compositio Mathematica}}, n 5({\bf 141}), 2005.

\bibitem[MR08]{MaRo}
M.~Matignon and M.~Rocher.
\newblock {\em Smooth curves having a large automorphism $p$-group in
  characteristic $p>0$}.
\newblock {\em {\rm Algebra and Number Theory}}, 2({\bf 8}), 2008.

\bibitem[Ped92]{Pe}
J.~P. Pedersen.
\newblock {\em A function field related to the Ree group}.
\newblock In H.~Stichtenothj and M.~A. Tsfasman, editors, {\em Coding Theory
  and Algebraic Geometry, Proceedings, Luminy 1991}. Springer, Berlin, 1992.

\bibitem[Roc09]{Ro}
M.~Rocher.
\newblock {\em Large $p$-group actions with a $p$-elementary abelian derived
  group}.
\newblock {\em {\rm Journal of Algebra}}, ({\bf 321}), 2009.

\end{thebibliography}

\end{document}